\documentclass{amsart}
\usepackage[utf8]{inputenc}
\usepackage{amssymb,amsmath,amsthm,amscd}
\usepackage{mathrsfs}
\usepackage{enumerate}
\usepackage{comment}
\usepackage{color}
\usepackage[all]{xy}

\usepackage{lmodern} %pour de plus belles polices
\def\sup{\mathop{\hbox{sup}}}

\def\R{\mathbb{R}}
\def\Q{\mathbb{Q}}
\def\N{\mathbb{N}}

\def\C{\mathbb{C}}

\renewcommand\Re{\mbox{Re\,}}

\newtheorem{thm}{Theorem}[section]
\newtheorem{lemma}[thm]{Lemma}
\newtheorem{lem}[thm]{Lemma}
\newtheorem{corollary}[thm]{Corollary}
\newtheorem{prop}[thm]{Proposition}
\newtheorem{defprop}[thm]{Definition-Proposition}
\newtheorem{question}[thm]{Question}

\theoremstyle{remark}
\newtheorem{rem}[thm]{Remark}

\theoremstyle{definition}
\newtheorem{dfn}[thm]{Definition}

\newcommand{\dom}{\mathrm{dom}}
\newcommand{\cS}{\mathcal{S}}

\newcommand{\Tr}{\mathrm{Tr}}

\newcommand{\cH}{\mathcal{H}}
\newcommand{\cR}{\mathcal{R}}

\newcommand{\cM}{\mathcal{M}}
\newcommand{\cN}{\mathcal{N}}
\newcommand{\cU}{\mathcal{U}}
\newcommand{\CB}{\mathcal{CB}}

\newcommand{\cL}{\mathcal{L}}
\newcommand{\MCB}{M_{\mathcal{CB}}A}
\newcommand{\nphi}{\mathfrak{n}_\varphi}
\newcommand{\mphi}{\mathfrak{m}_\varphi}

\title[Schur and Fourier multipliers of an amenable group]{Schur and Fourier multipliers of an amenable group acting on non-commutative $L^p$-spaces}
\author{Martijn Caspers}
\author{Mikael de la Salle}

 \address{M. Caspers, Laboratoire de Math\'ematiques, Universit\'e de Franche-Comt\'e, 16 Route de Gray, 25030 Besan\c con, France}
 \email{martijn.caspers@univ-fcomte.fr}
 \address{M. de la Salle, CNRS, Laboratoire de Math\'ematiques, Universit\'e de Franche-Comt\'e, 16 Route de Gray, 25030 Besan\c con, France}
 \email{mikael.de\_la\_salle@univ-fcomte.fr}

\date{\noindent \today.   {\it MSC2010}: 43A15, 46B08, 46B28, 46B70.  {\it Keywords}: non-commutative $L^p$-spaces, Schur multiplier, Fourier multipliers, amenability.\\ 
\\
The first author was supported by the ANR project: ANR-2011-BS01-008-01.\\
The scond author was partially supported by the ANR projects NEUMANN and OSQPI.
 }  

\begin{document}

\begin{abstract}
Consider a completely bounded Fourier multiplier $\phi$ of a locally
compact group $G$, and take $1 \leq p \leq \infty$. One can associate
to $\phi$ a Schur multiplier on the Schatten classes $\cS_p(L^2 G)$,
as well as a Fourier multiplier on $L^p(\cL G)$, the non-commutative
$L^p$-space of the group von Neumann algebra of $G$. We prove that the
completely bounded norm of the Schur multiplier is not greater than
the completely bounded norm of the $L^p$-Fourier multiplier. When $G$
is amenable we show that equality holds, extending a result by
Neuwirth and Ricard to non-discrete groups.

For a discrete group $G$ and in the special case when $p\neq 2$ is an even
integer, we show the following. If there exists a map between
$L^p(\cL G)$ and an ultraproduct of $L^p(\cM) \otimes \cS_p(L^2G)$
that intertwines the Fourier multiplier with the Schur multiplier,
then $G$ must be amenable. This is an obstruction to extend the
Neuwirth-Ricard result to non-amenable groups.
\end{abstract}

\maketitle

\section{Introduction}

This paper studies the close connection between Schur and Fourier multipliers of locally compact groups $G$. In particular, we are interested in the relation between the completely bounded norms of such multipliers.  

Given a bounded function $\phi:G \to \C$, the associated Fourier multiplier $T_\phi$, when it exists, is the unique weak-* continuous map on the von Neumann algebra $\cL G$ of $G$ extending $\lambda_g \mapsto \phi(g)\lambda_g$. The associated Schur multiplier $M_\phi$, when it exists, is the unique weak-* continuous map on $B(L^2G)$ extending $(a_{s,t})_{s,t \in G} \in \cS_2(L^2 G) \mapsto (\phi(s t^{-1}) a_{s,t})_{s,t \in G}$. Bo{\.z}ejko and Fendler proved \cite{BozFen84} that $M_\phi$ indeed determines a bounded map if and only if $T_\phi$ defines a completely bounded map, and in this case the completely bounded norms coincide. The function $\phi$ is then called a completely bounded Fourier multiplier and the space of such $\phi$ is denoted $\MCB(G)$. Note that $T_\phi$ is the restriction of $M_\phi$ to $\cL G \subset B(L^2G)$.

In this paper we investigate the Bo{\.z}ejko-Fendler result for $L^p$-multipliers. When $G$ is discrete and $\phi \in \MCB(G)$, there is no technical difficulty for defining the Fourier multiplier $T_\phi^p:L^p(\cL G) \to L^p(\cL G)$ and the Schur multiplier $M_\phi^p:\cS_p(\ell^2 G) \to \cS_p(\ell^2 G)$ (they are just the extension of $T_\phi$/ restriction of $M_\phi$). When $G$ is not discrete, the definition of the Schur multiplier $M_\phi^p:\cS_p(L^2 G) \to \cS_p(L^2 G)$ is known (see Subsection \ref{subsection=Schur}), and we define in Subsection \ref{subsection=Fourier} the $L^p$-Fourier multiplier $T_\phi^p:L^p(\cL G) \to L^p(\cL G)$ spatially \cite{Con}, \cite{Hilsum}.

We are interested in the following question.

\begin{question}\label{Que-One}
Let $G$ be a locally compact group. Is it true that,
\[
\Vert T_\phi^p \Vert_{\CB(L^p(\cL G))} = \Vert M_\phi^p \Vert_{\CB(\cS_p(L^2G))}.
\]
for all completely bounded Fourier multipliers $\phi$ of $G$?
\end{question}

\vspace{0.3cm}

There are several motivations to study this problem. One is that a positive answer would imply that for a discrete group $G$, the property ${\rm AP}^{Schur}_{p,cb}$ considered in \cite{LafSal}, see also \cite{Laa}, is an invariant of the group von Neumann algebra of $G$. Furthermore, one is often interested in strict estimates of the norm of the transfered Fourier multiplier, see for example \cite[Section 7]{ChenXuYin}. 

\vspace{0.3cm}

 In \cite{NeuRic}, Neuwirth and Ricard studied Question \ref{Que-One} for discrete groups. They noted that the inequality $\geq$ always holds, and they proved the other inequality in the case $G$ is amenable. This is an $L^p$-version of the Bo{\.z}ejko-Fendler result. Their proof relies on the fact that the amenability of $G$ allows them to construct a completely isometric embedding of $L^p(\cL G)$ into an ultrapower of $\cS_p(\ell^2 G)$ that intertwines Fourier and Schur multipliers.

Marius Junge (personal communication) pointed out to the second-named author that a positive answer to Question \ref{Que-One} would also follow from the existence of a completely isometric embedding of $L^p(\cL G)$ into an ultraproduct of $L^p(\cM_n) \otimes \cS_p(L^2 G)$ for some net of von Neumann algebras $\cM_n$, that intertwines Fourier and Schur multipliers. He suggested that some weaker approximation property (as exactness) might provide such an embedding. Our first result (Theorem \ref{Thm-CounterExample}) partially answers his suggestion negatively: if $p\neq 2$ is an even integer, $G$ is a discrete group, and such an embedding exists, then $G$ is amenable. This bad news was the motivation to study Question \ref{Que-One} for general locally compact groups.

The main part of the present paper is to extend the Neuwirth-Ricard result to arbitrary locally compact groups: we first prove in Theorem \ref{Thm-TransferenceEstimate} by a transference technique that the inequality $\geq$ always holds. We then prove the inequality $\leq$ for $G$ amenable in Theorem \ref{thm=embedding_in_ultraproduct} and its Corollary \ref{Cor-MainEquality}. Our proof is close to the proof given in \cite[Section 3]{NeuRic}, but we encounter several technical issues that we have to overcome, mainly in the case of non-unimodular groups.

Let us explain a bit more in details the technical complications that we meet.
Recall \cite{PisXu}   that to a von Neumann algebra $\cM$ one can associate non-commutative $L^p$ spaces $L^p(\cM)$ for $1 \leq p \leq \infty$. There are several different ways to define $L^p(\cM)$, but they all yield   (completely) isometric spaces. When $\cM = B(\cH)$, $L^p(\cM)$ is most naturally realized as the Schatten class $\cS_p(\cH)$ (which is contained in $B(\cH)$). 
When $\cM$ carries a normal faithful finite trace $\tau$ (for example $\cM = \cL G$ for a discrete group $G$), the natural choice is to see $L^p(\cM)$ as the completion of $\cM$ for the norm $\|x\|_p = \tau(|x|^p)^{1/p}$ ($L^p(\cM)$ contains $\cM$). When $G$ is a locally compact group, there is a natural weight on $\cL G$, the Plancherel weight, which is a trace if and only if $G$ is unimodular. Although $\cL G$ is often semi-finite, it is more natural to view the $L^p$-spaces of $\cL G$ as contructed from this weight. As is well-known, the non-commutative $L^p$-spaces then have the property that   $L^p$ and $L^q$ do not naturally intersect for $p \neq q$. It turns out that the most convenient way to define the $L^p$-Fourier multiplier is by realizing $L^p(\cL G)$ in the Connes-Hilsum construction (Subsection \ref{subsect=Connes-Hilsum}) associated to the Plancherel weight on $\cL G'  = \cR G$ (the von Neumann algebra of $L^2G$ associated to the right translations). Indeed, this realizes $L^p(\cL G)$ as a space of unbounded operators on $B(L^2 G)$, and the naive definition of $T_\phi^p$ would just be the restriction of $T_\phi$ to $L^p(\cL G)$. Since in general (\emph{i.e.} unless $G$ is unimodular) $L^p(\cL G) \cap B(L^2(G)) = \{0\}$ this is a bit to naive, and the precise definition along these lines is given in the Definition-Proposition \ref{dfprop=def__Tphip}.

An important observation we make is that although $L^p(\cL G)$ might not contain nonzero bounded operators, its corners in fact lie in the Schatten class $\cS_p$. More precisely, for $F \subset G$ relatively compact we can (carefully) define the corner $P_F x P_F$ of any element $x \in L^p(\cL G)$, and we prove in Theorem \ref{thm=contraction_well_defined} that $P_F x P_F \in \cS_p(L^2F)$ with norm at most $|F|^{1/p}\|x\|_{L^p(\cL G)}$. When $G$ is amenable, one then chooses a F\o{}lner net $F_n \subset G$ and consider the maps $x \mapsto |F_n|^{-1/p} P_{F_n} x P_{F_n}$.

In all the definitions and statements we give, $L^p(\cL G)$ will be the Connes-Hilsum space, but the theory of interpolation \cite{TerpII}, \cite{Izumi} of non-commutative $L^p$-spaces associated to a weight plays an important role in our proofs. This is a bit technical. Therefore, to make the paper readable we collect (and prove when necessary) all the results we need from interpolation in Theorem \ref{thm=interpolation_BB}, and in the proofs we use this theorem as an abstract black box.
\vspace{0.3cm}

The structure of the paper is as follows. Section \ref{Sect-Impossible}, which is independent from the rest of the paper, gives a sufficient (and necessary) condition for a discrete group to be amenable in terms of intertwiners of Schur and Fourier multipliers. Section \ref{Sect-Preliminaries} recalls the necessary results and fixes the notation. In Section \ref{Sect=transference}, we prove that the inequality $\geq$ in Question \ref{Que-One} holds for any locally compact group. Section \ref{Sect=amenable} gives the affirmative answer to Question \ref{Que-One} for amenable groups. The last Section \ref{Sect=interpolation} on interpolation states and then proves all the results needed in the first sections on interpolation of non-commutative $L^p$-spaces.

\section{Characterizing amenability by intertwining Schur and Fourier multipliers}\label{Sect-Impossible}
 
In this section we prove the following characterization of the amenability of a discrete group. It takes away our hope to give a direct positive answer to Question \ref{Que-One} for non-amenable discrete groups.

Recall that the Schur multiplier $M_\phi^p$ and the Fourier multiplier $T_\phi^p$ were defined in the introduction for discrete groups. We define them more generally in Section \ref{Sect-Preliminaries}. 

\begin{thm}\label{Thm-CounterExample}
 Let $\Gamma$ be a discrete group, let $\cH$ be a Hilbert space and let $p\geq 4$ be an even integer. Let $\cU$ be a non-trivial ultrafilter on some set. Assume that there exists an isometric embedding
 \begin{equation}\label{Eq-IntertwinerCounter}
j_p: L^p(\cL \Gamma) \to \prod_{\mathcal U} \cS_p(\ell^2 \Gamma \otimes
  \cH)
\end{equation}
such that for every $\phi :\Gamma \to \C$ of finite support,
  $j_p \circ T_\phi^p = (\prod_{\cU} M_\phi^{p} \otimes id) \circ j_p$. Then, $\Gamma$ is amenable.

Conversely, if $\Gamma$ is amenable such an embedding exists for all $1 \leq p \leq \infty$ and we can take $\cH=\C$.
\end{thm}
\begin{rem} 
In fact the proof below shows more generally that the same 
conclusion holds if there is such an embedding of $L^p(\cL(\Gamma))$
into $\prod_{\mathcal U} L^p( B(\ell^2 \Gamma)   \otimes  
\cM_\alpha)$ with a net $(\cM_\alpha,\tau_\alpha)$ of semi-finite von Neumann algebras.
%%  The QWEP hypothesis is to ensure that $M_\phi^4
%% \otimes id$ is bounded on $L^4( B(\ell^2 \Gamma)  \otimes  
%% M_\alpha)$ (see \cite[Remark 1.1]{LafSal}).
\end{rem} 
\begin{proof}
To keep the notation simple, we consider the case $p=4$ and write $j$ for $j_4$. The same proof works for every $p \in 2\N$. 

Fix $\gamma \in \Gamma$, and let $(\Xi_\alpha)$ be a representative of
$j(\lambda(\gamma))$. Express that $j \circ T_\phi^4 = (\prod_{\cU} M_\phi^{p} \otimes id) \circ j$ for $\phi = \delta_\gamma$ a Dirac mass point. This
implies that $j(\lambda(\gamma))=((M_{\delta_\gamma}^4 \otimes id)
(\Xi_\alpha))_{\alpha}$, so that we can assume that
$(M_{\delta_\gamma} \otimes id) (\Xi_\alpha) = \Xi_\alpha$ for all
$\alpha$. In other words, if we see $\Xi_\alpha$ as an infinite matrix
with values in $\cS_p(\cH)$, its $(s,t)$ entry is zero unless
$st^{-1}=\gamma$. Equivalently, there is an element $\xi \in
\ell^p(\Gamma;\cS_p(\cH))\subset \cS_p(\ell^2 \Gamma \otimes \cH)$ such that
$\Xi_\alpha = (\lambda(\gamma) \otimes id) \xi$. Let us denote this
$\xi$ by $\xi_\gamma^{(\alpha)}$. Using that $1 =
\|\lambda(\gamma)\|_{L^4(\cL(G))} = \lim_{\alpha, \mathcal U}
\|\xi_\gamma^{(\alpha)}\|_{\ell^4(\Gamma; \cS_4(\cH))}$, we can as well assume that
$\|\xi_\gamma^{(\alpha)}\|_{\ell^4(\Gamma; \cS_4(\cH))}=1$. Since $\gamma \in\Gamma$
was arbitrary, we can choose such a representative for all $\gamma$.

Let $\eta_\gamma^{(\alpha)}$ be the unit vector in $\ell^4 \Gamma$
defined by $\eta_\gamma^{(\alpha)} =
(\|\xi_\gamma^{(\alpha)}(s)\|_{\cS_4(\cH)})_{s \in \Gamma}$. We claim that,
\begin{equation}\label{eq=invariant}
\forall \gamma,s \in \Gamma, \lim_{\alpha, \mathcal U} \|\lambda_s 
(\eta_\gamma^{(\alpha)})^2 - (\eta_\gamma^{(\alpha)})^2\|_{\ell^2\Gamma} = 0.
\end{equation}
In particular, $\lambda$ has almost invariant vectors in
$\ell^2\Gamma$, which implies that $\Gamma$ is amenable.

To prove the claim, take $x = \sum_\gamma x_\gamma \lambda_\gamma$
arbitrary in the group algebra $\C[\Gamma]$ with $x_\gamma \geq 0$. Then, we find $x^* x = \sum_s (\sum_\gamma
\overline{x_\gamma} x_{\gamma s}) \lambda_s$, so that
\[
\|x\|_{L^4(\cL G)}^4 = \sum_{s,\gamma,\widetilde \gamma} x_\gamma \overline{x_{\gamma s}} \overline{x_{\widetilde \gamma}} x_{\widetilde \gamma s}.
\]
To compute $\|j(x)\|^4_{\prod_{\mathcal U} \cS_4(\ell^2 \Gamma \otimes
  \cH)} $ and to simplify the notation, we shall denote,
for $s \in\Gamma$ and $\xi \in \ell^p(\Gamma;\cS_p(\cH))$, $s \cdot \xi=
(\xi(s^{-1}t))_{t \in \Gamma}$. We shall also use the relation $s \cdot
\xi = (\lambda_s \otimes id) \xi (\lambda_s^{\ast} \otimes id)$. Hence,
\[
j(x)^* j(x) = \left( \sum_s (\lambda_s \otimes id) (\sum_\gamma \overline{x_\gamma}x_{\gamma s}(s^{-1} \cdot \xi_\gamma^{\alpha})^* \xi_{\gamma s}^\alpha)\right)_{\alpha},
\]
so that
\begin{eqnarray*} 
  \|j(x)\|_{\prod_{\mathcal U} \cS_4(\ell^2 \Gamma \otimes
  \cH)} ^4 &=& \lim_{\alpha, \mathcal U} \sum_{s,\gamma,\widetilde \gamma} x_\gamma \overline{x_{\gamma s}} \overline{x_{\widetilde \gamma}} x_{\widetilde \gamma s} \Tr\left(\xi_{\gamma s}^{(\alpha)*} s^{-1}.(\xi_\gamma^{(\alpha)} \xi_{\widetilde \gamma}^{(\alpha)*}) \xi_{\widetilde \gamma s}^{(\alpha)} \right).
%\\&=& \lim_{\mathcal U} \sum_{s,\gamma,\widetilde \gamma} x_\gamma \overline{x_{\gamma s}} \overline{x_{\widetilde \gamma}} x_{\widetilde \gamma s} Tr\left(s^{-1}.(\xi_\gamma^{(\alpha)} \xi_{\widetilde \gamma}^{(\alpha)*}) \xi_{\widetilde \gamma s}^{(\alpha)}\xi_{\gamma s}^{(\alpha)*}\right).
\end{eqnarray*}
Here, $\Tr$ is the natural (semi-finite) trace on $\ell^\infty(\Gamma; B(\cH))$. 
By $x_\gamma\geq 0$ and H\"older's inequality, we have 
\begin{eqnarray*}  \|j(x)\|_{\prod_{\mathcal U} \cS_4(\ell^2 \Gamma \otimes
  \cH)} ^4 & = & \sum_{s,\gamma,\widetilde \gamma} x_\gamma {x_{\gamma s}} {x_{\widetilde \gamma}} x_{\widetilde \gamma s} \lim_{\alpha, \cU} \Re \: \Tr\left(s^{-1}.(\xi_\gamma^{(\alpha)} \xi_{\widetilde \gamma}^{(\alpha)*}) \xi_{\widetilde \gamma s}^{(\alpha)}\xi_{\gamma s}^{(\alpha)*}\right)\\
& \leq & \sum_{s,\gamma,\widetilde \gamma} x_\gamma {x_{\gamma s}} {x_{\widetilde \gamma}} x_{\widetilde \gamma s} 1= \|x\|_{L^4(\cL G)}^4.
\end{eqnarray*}
Hence if $\gamma,\widetilde \gamma,\gamma s,\widetilde \gamma s$ belong to the support of $x$,
\[\lim_{\alpha,\cU} \Re \: \Tr\left(s^{-1}.(\xi_\gamma^{(\alpha)} \xi_{\widetilde \gamma}^{(\alpha)*}) \xi_{\widetilde \gamma s}^{(\alpha)}\xi_{\gamma s}^{(\alpha)*}\right) = 1,\]
so that by the uniform convexity Lemma \ref{lemma=uniform_convexity} below 
\begin{equation}\label{eq=limit_terme_a_terme}
\lim_{\alpha, \mathcal U}\|\lambda_s (\eta_{\widetilde \gamma s}^{(\alpha)})^2 - (\eta_{\gamma}^{(\alpha)})^2\|_{\ell^2(\Gamma)} = 0.
\end{equation}
Since $x$ was arbitrary, \eqref{eq=limit_terme_a_terme} holds for all
$s,\gamma,\widetilde \gamma \in \Gamma$. This proves
\eqref{eq=invariant}, and hence that $\Gamma$ is amenable.

The converse was proved in \cite{NeuRic}.
\end{proof}

\begin{lemma}\label{lemma=uniform_convexity}
Let $\Gamma$ be a discrete group and let $\cH$ be a Hilbert space. 
For all $\epsilon >0$, there exists $\delta>0$ such that the following
holds: Let $s \in \Gamma$ and $\xi_i \in \ell^4(\Gamma;\cS_4(\cH))$, $i=1\dots 4$
be unit vectors, and denote by $\eta_i$ the unit vector
$(\|\xi_i(t)\|_{\cS_4(\cH)})_{t \in \Gamma}$ in $\ell^4 \Gamma$. If
\[
{\rm Re}\: \Tr\left(s^{-1}.(\xi_1 \xi_2) \xi_3 \xi_4\right) \geq 1-\delta,
\]
then $\|\lambda_s \eta_3^2 - \eta_1^2\|_{\ell^2\Gamma} \leq \epsilon$.
\end{lemma}
\begin{proof}
In the proof we write $o(1)$ for a quantity that goes to $0$
as $\delta$ goes to $0$.  By H\"older's inequality in $\cS_4(\cH)$,
\begin{eqnarray*} 
1-\delta\leq\Re\: \Tr\left(s^{-1}.(\xi_1 \xi_2) \xi_3
  \xi_4\right)&\leq& \langle \lambda_{s^{-1}}(\eta_1 \eta_2), \eta_3
  \eta_4\rangle\\ & \leq & \|\eta_1 \eta_2\|_{\ell^2\Gamma} \|\eta_3
  \eta_4\|_{\ell^2\Gamma} \leq 1.
\end{eqnarray*} 
In particular,
$\langle \eta_1^2,\eta_2^2\rangle = \|\eta_1 \eta_2\|_{\ell^2\Gamma}^2 \geq
(1-\delta)^2$, which clearly implies that $\|\eta_1^2 - \eta_2^2\|_{\ell^2\Gamma} =
o(1)$. The elementary inequality $\|a-b\|_{\ell^4\Gamma} \leq \sqrt{\|a^2 - b^2\|_{\ell^2\Gamma}}$ valid for all $a,b:\Gamma \to \R_+$ in turn implies $\|\eta_1 -
\eta_2\|_{\ell^4\Gamma} = o(1)$. Similarly, $\|\eta_3 - \eta_4\|_{\ell^4\Gamma}  =o(1)$.  By the
first line of the series of inequalities above, we get $\langle
\lambda_{s^{-1}} \eta_1^2, \eta_3^2 \rangle =1+o(1)$, which indeed
implies $\|\lambda_s \eta_3^2 - \eta_1^2\|_{\ell^2\Gamma} = o(1)$.
\end{proof}

\section{Preliminaries} \label{Sect-Preliminaries}

We collect the necessary results and fix the notation. In particular, we introduce $L^p$-Fourier multipliers for any locally compact group.

\subsection{General notation} 
For a Hilbert space $\cH$, $\cS_p(\cH)$ denotes the Schatten class, i.e. the non-commutative $L^p$-space associated with $B(\cH)$. We use brackets $[\: ]$ to denote the closure of an operator and $\cdot$ for the strong product of (unbounded) operators. Recall that $a \cdot b$, when it exists, is the closure of the operator defined on $\{x \in \dom(b), b(x) \in \dom(a)\}$ by $x \mapsto a(b(x))$. For $\varphi$ a normal, semi-finite, faithful weight on a von Neumann algebra $\cM$, we set $\nphi = \{ x \in \cM \mid \varphi(x^\ast x ) < \infty \}$ and $\mphi = \nphi^\ast \nphi$. 

\subsection{Integral operators}
Let $(X,d x)$ be a measure space. A (bounded) integral operator $A$ on $L^2(X,d x)$ is a bounded operator for which there is a measurable function $(s,t) \mapsto A_{s,t}$ (called the kernel of $A$) on $X \times X$ such that, for all $\xi \in L^2(X,d x)$, 
\begin{itemize}
\item $t \mapsto A_{s,t} \xi(t) \in L^1(X,d x)$ for almost every $s \in X$.
\item $A \xi (s) = \int_X A_{s,t} \xi(t)$ for almost every $s \in X$.
\end{itemize}
By Fubini the kernel is only defined up to an almost everywhere zero function.
The set of integral operators is a self-adjoint subspace of $B(L^2(X))$, and the kernel of the adjoint is given by $(A^*)_{s,t} = \overline{A_{t,s}}$.
The Schatten $\cS_2$-class consists of all integral operators with kernel belonging to $L^2(X \times X)$, and for such operator, 
\[ \|A\|_{\cS_2(L^2X)} = \left(\iint |A_{s,t}|^2 ds dt\right)^{1/2}.\]
More generally if $A,B \in \cS^2(L^2X)$,
\begin{equation}\label{eq=formula_for_trace} \Tr(AB) =\iint A_{s,t} B_{t,s} dt ds.\end{equation}

\subsection{Schur multipliers}\label{subsection=Schur}
%% For any Hilbert space $\cH$, the space $\cS_2(\cH)$ can be identified with $\cH \otimes \overline{\cH}$ by means of   $\theta_{\xi, \eta} \mapsto \xi \otimes \overline{\eta}$. Here,  $\theta_{\xi, \eta}$ is the Hilbert-Schmidt operator given by  $\theta_{\xi, \eta}(v) = \langle v, \eta \rangle \xi$.

Let $(X, \mu)$ be a measure space and let $\psi: X \times X \rightarrow \mathbb{C}$ be an essentially bounded measurable function. 
We identify $\cS_2(L^2 X )$ linearly with $L^2(X \times X)$ by the previous paragraph. Then, we obtain a bounded map
\[
 L^2(X \times X) \rightarrow  L^2(X \times X): (a_{x,y})_{x,y \in X} \mapsto (\psi(x,y) a_{x,y})_{x,y \in X},
\] 
which canonically determines a bounded map on $\cS_2(L^2 X)$. If it maps $\cS_2(L^2 X)\cap \cS_p(L^2 X)$ to  $\cS_p(L^2 X)$ and it extends boundedly to $\cS_p(L^2X)$, it will be called a (Schur) multiplier of $\cS_p(L^2X)$.

We will use the following result, which is a minor modification of \cite[Theorem 1.19]{LafSal}. 
\begin{thm}\label{thm=mult_de_symb_continu}
  Let $\mu$ be a Radon measure on a locally compact space
  $X$, and $\psi:X \times X \to \C$ a continuous function. Let $1 \leq p
  \leq \infty$ and $K>0$. The following are equivalent:
\begin{enumerate}[(i)] 
\item \label{asser=norm_locCompact} $\psi$ defines a bounded multiplier
  on $\cS_p(L^2(X,\mu))$ with norm less than $K$.
\item \label{asser=norm_parties_semifinies} For every  $\sigma$-finite measurable subset
  $X_0$ in $X$,  $\psi$ restricts to a bounded  multiplier   on $\cS_p(L^2( X_0, \mu))$ with norm
  less than $K$.
\item \label{asser=norm_parties_finies} For any finite subset
  $F=\{x_1,\dots,x_N\}$ in $X$ belonging to the support of $\mu$, the
  multiplier $(\psi(x_i,x_j))$ is bounded on $\cS_p(\ell^2 F)$ with norm
  less than $K$.
\end{enumerate}
\end{thm}
\begin{proof}
The equivalence of (\ref{asser=norm_locCompact}) and (\ref{asser=norm_parties_finies}) of  Theorem \ref{thm=mult_de_symb_continu} was proved in \cite[Theorem 1.19]{LafSal} under the additional assumption that $\mu$ is $\sigma$-finite. In the current theorem, this implies the equivalence of (\ref{asser=norm_parties_semifinies}) and (\ref{asser=norm_parties_finies}). The implication (\ref{asser=norm_locCompact}) implies (\ref{asser=norm_parties_semifinies}) is trivial. Assume (\ref{asser=norm_parties_semifinies}). Let $x \in \cS_p(L^2(X,\mu))$. The supports of $x$ and $x^\ast$ are $\sigma$-finite projections in $B(L^2X)$. So, let $X_0 \subset X$ be $\sigma$-finite such that the support projections of $x$ and $x^\ast$ project onto spaces contained in $L^2(X_0, \mu)$. Then, the multiplier $\psi$ applied to $x$ is equal to the restriction of $\psi$ to $X_0 \times X_0$ applied to $x$. Hence, (\ref{asser=norm_locCompact}) follows.  
 \end{proof}
 
Let $G$ be a locally compact group. Let $\phi \in \MCB(G)$ and set $\check{\phi} \in L^\infty(G \times G)$ by $\check{\phi}(s,t) = \phi(s t^{-1})$. Then, $\check{\phi}$ is a Schur multiplier acting on $\cS_p(L^2G)$, see \cite{LafSal} for details. We will denote this map by $M_\phi^p$.

\subsection{The von Neumann algebra of $G$.}\label{subsect=groupvna}
Let $G$ be a locally compact group equipped with a left Haar
measure, and denote by $\Delta:G\to (0,\infty)$ the modular
function. Recall (see \cite[p. 65]{TakII}) that $\Delta$ is the group morphism satisfying the following equations for all compactly supported continuous function $f:G \to \C$~:
\begin{equation}\label{eq=modular1} \int f(ts) dt = \Delta(s)^{-1} \int f(t) dt,\end{equation}
\begin{equation}\label{eq=modular2} \int f(t^{-1}) dt = \int f(t) \Delta(t)^{-1} dt.\end{equation}

 Denote by $\lambda$ and $\rho$ the left and right regular
representations of $G$~: for $s \in G$, $\lambda_s$ and $\rho_s$ are
unitaries on $L^2(G)$ given by $\lambda_s \xi(t) = \xi(s^{-1}t)$ and
$\rho_s \xi(t) = \sqrt{\Delta(s)} \xi(ts)$. The (left) von Neumann
algebra of $G$ is defined by $\cL G = \lambda(G)''$. Its commutant is
$\rho(G)''$, the right von Neumann algebra of $G$. If $f \in L^1(G)$, the formulas $\langle \lambda(f)\xi,\eta\rangle = \int f(s) \langle \lambda_s \xi,\eta\rangle ds$ and $\langle \rho(f)\xi,\eta\rangle = \int f(s) \langle \rho_s \xi,\eta\rangle ds$ define operators $\lambda(f) \in \cL G$ and $\rho(f) \in \rho(G)''$ that are integral operators on $L^2(G)$ with kernel
%the convolution on the left by $f$ defines a bounded operator on $L^2(G)$ denoted by $\lambda(f)$. This operator in fact belongs to $\cL G$ and is an integral operator with kernel 
\begin{equation}\label{eq=kernel_lambdaf} (\lambda(f))_{s,t} = \Delta(t^{-1}) f(st^{-1})\ ,\ (\rho(f))_{s,t} = \sqrt{\Delta(s^{-1}t)} f(s^{-1}t).\end{equation}
As a consequence, $\lambda(f)^\ast= \lambda(f^\ast)$ and $\rho(f)^\ast=\rho(f^\ast)$ where $f^\ast \in L^1(G)$ is given by
\begin{equation}\label{eq=def_of_xistar}f^\ast(s) = \overline{f(s^{-1})} \Delta(s^{-1}).\end{equation} 

When $G$ is discrete, $\cL G$ is finite and carries a natural
trace. More generally when $G$ is unimodular $\cL G$ carries a natural
semifinite trace. In general, although $\cL G$ might be semifinite, it
is not equipped with a natural trace, but rather with its natural
weight, given by $\varphi (x^*x) = \|f\|_{L^2G}^2$ if there is $f \in
L^2G$ such that $x \xi = f \ast \xi$ for all $\xi \in L^2G$ and $\varphi(x^*x)=\infty$ otherwise.
%x =\lambda(f)$ with $f \in L^(G)
%\cap L^2(G)$% and $\varphi(x^*x)=\infty$ otherwise (A VERIFIER) 
(it is an easy exercice to check that $\varphi$ is a trace if and only
if $G$ is unimodular). It is convenient to work with this weight.  A
natural weight on the commutant $\rho(G)''$ of $\cL G$ in $B(L^2 G)$
is given by $\psi(x^*x) = \|f\|_{L^2G}^2$ if $x = \rho(f)$ with $f
\in L^2G$ and $\psi(x^*x)=\infty$ otherwise.

%% We recall that the Fourier algebra $A(G)$ of $G$ is isomorphic to the
%% predual of $\cL G$. I have the feeling that the most convenient
%% realisation of the non-commutative $L^p$ space of $\cL G$ is the
%% Connes-Hilsum construction. See \cite{MR677418}.
\subsection{The non-commutative $L^p$-space of $\cL G$.} \label{subsect=Connes-Hilsum}
We recall the Connes-Hilsum construction \cite{Con}, \cite{Hilsum} of $L^p(\cL G)$ in the particular case when $\cL G$ is the von Neumann algebra of $G$. The space $D(L^2 G,\psi)$ of $\psi$-bounded elements of $L^2G$ is the set of functions $\xi \in L^2G$ such that $f \mapsto \xi \ast f$ is bounded on $L^2G$. We will also denote this operator by $\lambda(\xi)$. Given a weight $\omega$ on $\cL G$, the spatial derivative $d\omega/d\psi$ is the unique positive self-adjoint operator such that $\omega(\lambda(\xi) \lambda(\xi)^\ast) = \| (d\omega/d\psi)^{1/2} \xi\|_{L^2G}^2$ for all $\psi$-bounded $\xi$. For example, \eqref{eq=def_of_xistar} gives that $d\varphi/d\psi = \Delta$. The spatial derivative of an element $\omega$ of the predual is defined by $d \omega/d\psi = u d |\omega|/d\psi$ where $\omega = u |\omega|$ is the polar decomposition of $\omega$. The set $L^1(\cL G)$ defined as $\{ d\omega/d\psi,\omega \in \cL G_*\}$ is then a linear space (the sum being the closure of the sum), and we denote $\int (d\omega/d\psi) d\psi = \omega(1)$. For an arbitrary $1 \leq p<\infty$, $L^p(\cL G)$ is defined as the set of closed densily defined operators $T$ with polar decomposition $T=u|T|$ satisfying $u \in \cL G$, $|T|^p \in L^1(\cL G)$, and one denotes $\|T\|_{L^p(\cL G)} = (\int |T|^p d\psi)^{1/p}$. In general the non-zero elements of $L^p(\cL G)$ are not bounded operators. Hilsum \cite{Hilsum} proved that the sum of two elements of $L^p(\cL G)$ is densely defined, closable, that its closure belongs to $L^p(\cL G)$, and that for this linear structure if $p\geq 1$, $L^p(\cL G)$ is a Banach space for the norm $\| \cdot \|_{L^p(\cL G)}$. Moreover, for $0\leq p,q,r \leq \infty$ with $1/r=1/p+1/q$ and $a \in L^p(\cL G)$ and $b \in L^q(\cL G)$, $ab$ is closable and its closure (still denoted by $ab$) belongs to $L^r(\cL G)$   and this product is associative. Lastly, if $r=1$, $\int a b d\psi=\int b a d\psi$, and if $p\neq \infty$ the pairing $\langle a,b\rangle=\int a b d\psi$ realizes $L^q(\cL G)$ isometrically as the dual of $L^p(\cL G)$.

Apart from what we just recalled, we will use some facts from \cite[Proposition 11]{Hilsum} that we collect in the following proposition.

\begin{prop}\label{prop=reminders_hilsum} We have the following properties.
\begin{enumerate}
\item We have $u D(L^2G,\psi)\subset D(L^2G,\psi)$ for every $u \in \cL G$.
\item If $x \in L^p(\cL G)$ with $2 \leq p \leq \infty$, then $D(L^2G,\psi) \subset \dom(x)$.
\item Let $1 < p \leq \infty$. For any $\xi,\eta \in D(L^2G,\psi)$ there is a bounded linear map $\omega_{\xi,\eta}^p: L^p(\cL G) \to \C$ satisfying $\omega_{\xi,\eta}^p(x) = \langle |x|^{1/2} \xi,|x|^{1/2} u^* \eta \rangle$ if $x= u |x|$ is the polar decomposition of $x \in L^p(\cL G)$.
\end{enumerate}
\end{prop}
As a direct consequence, using the inclusion $L^2 F   \subset D(L^2G,\psi)$ and the closed graph theorem we get the following.
\begin{prop}\label{prop=def_of_PFxPF}
Let $F \subset G$ be a relatively compact Borel subset with positive measure, and $P_F:L^2G  \to L^2 F$ the orthogonal projection.
\begin{enumerate}
\item If $p \geq 2$, $L^2F \subset \dom(x)$ for every $x \in L^p(\cL G)$, and $x \in L^p(\cL G) \to x P_F \in B(L^2 G)$ is a linear bounded map. 
\item If $1 \leq p \leq \infty$ there is a bounded linear map $L^p(\cL G) \to B(L^2 F)$, that maps an element $x =u |x| \in L^p(\cL G)$ to $ (|x|^{1/2}u^*P_F)^* |x|^{1/2} P_F$. We abusively denote this map by $x \mapsto P_F x P_F$.
\end{enumerate}
\end{prop}

 \subsection{Completely bounded maps between von Neumann algebras} Let $\cM \subset B(\cH)$ a von Neumann algebra, $\psi$ a weight on $\cM'$ and $L^p(\cM,\psi)$ the spatial non-commutative $L^p$-space. The commutant of $M_n \otimes \cM$ in $M_n \otimes B(\cH)$ is $1 \otimes \cM'$, and is therefore equipped with the weight $1 \otimes \psi$. This naturally identifies $\cS_p^n \otimes L^p(\cM,\psi)$ with $L^p(M_n \otimes \cM, 1 \otimes \psi)$. We use this identification to define a Banach space structure on $\cS_p^n \otimes L^p(\cM,\psi)$.

If $\cM\subset B(\cH),\cN \subset B(\cH')$ are von Neumann algebras and $\psi,\psi'$ are weights on $\cM'$, $\cN'$, a bounded map $u:L^p(\cM,\psi) \to L^p(\cN,\psi')$ between the non-commutative $L^p$-spaces (Connes-Hilsum construction) is called completely bounded if $\|u\|_{cb}:=\sup_n \|u_n\|<\infty$. Here $u_n = id \otimes u : \cS_p^n \otimes L^p(\cM,\psi)\to \cS_p^n \otimes L^p(\cN,\psi')$. When $\cM$ and $\cN$ are semifinte, one can check that this definition agrees with the natural operator space structure on non-commutative $L^p$-spaces given in \cite{MR1648908}. 

\subsection{Definition of $L^p$ Fourier multipliers}\label{subsection=Fourier}Recall that the Fourier algebra  $A(G) = \left\{ \varphi\colon s\mapsto \langle \lambda_s \xi,\eta\rangle, \xi,\eta \in L^2 G\right\}$ coincides with the predual of $\cL G$ through the pairing $\langle \varphi,\lambda(f) \rangle = \int \varphi(s) f(s) ds$ for every $f \in L^1(G)$. We record here the explicit isomorphism between the Fourier algebra $A(G)$ and $L^1(\cL G)$. 
\begin{lemma}\label{lemma=isom_AG_L1} An element $\varphi \in A(G)$ corresponds (isometrically) to $x_\varphi \in L^1(\cL G)$ that satisfies $P_F x_\varphi P_F = (\varphi(ts^{-1}))_{s,t \in F}$.
\end{lemma}
\begin{proof}
Start by considering the case when $\varphi$ is positive definite, so that the corresponding element of $\cL G_*$ is positive. By the definition of $L^1(\cL G)$ and the spatial derivative, $x_\varphi$ is the positive self-adjoint operator characterized by $\|x_\varphi^{1/2} \xi\|^2 =\langle \varphi, \lambda(\xi)\lambda(\xi)^*\rangle$ for all $\xi \in L^2G$ that are $\psi$-bounded. By definition of $P_F x_\varphi P_F$, and the fact that $L^2 F \subseteq L^1 G$, we get 
\[ \langle P_F x_\varphi P_F \xi,\xi\rangle = \iint_{F \times F} \varphi(t s^{-1}) \xi(t) \overline{\xi(s)} ds dt,\]
and by polarization 
\[ \langle P_F x_\varphi P_F \xi,\eta\rangle = \iint_{F \times F} \varphi(t s^{-1}) \xi(t) \overline{\eta(s)} ds dt.\]
This proves the lemma when $\varphi$ is positive definite. The general case follows by the linearity of $\varphi \mapsto P_F x_\varphi P_F$.
\end{proof}

\subsection{Fourier multipliers}
Given a function $\phi \in \MCB(G)$, we would like to define the Fourier multiplier on $L^p(\cL G)$ as the restriction to $L^p(\cL G)$ of $M_\phi$. This is possible when $G$ is unimodular because $L^p(\cL G) \cap B(L^2 G)$ is dense in $L^p(\cL G)$. When $G$ is not unimodular (and $p \neq \infty$) this is more problematic: $L^p(\cL G) \cap B(L^2 G) = \{0\}$, whereas $M_\phi$ is only defined on $B(L^2G)$. The definition is therefore done with the help of Proposition \ref{prop=def_of_PFxPF}.
\begin{defprop}\label{dfprop=def__Tphip} Let $\phi \in \MCB(G)$. There is a unique completely bounded linear map $T_\phi^p: L^p(\cL G) \to L^p(\cL G)$ that satisfies $P_F T_\phi^p(x) P_F = M_\phi (P_F x P_F)$. It has completely bounded norm less than $\|\phi\|_{\MCB(G)}$. This map is called the Fourier multiplier with symbol $\phi$.
\end{defprop}
\begin{proof}
Let us first prove that $T_\phi^p$ exists and is bounded. For $p=\infty$ this is obvious. 
Assume $p=1$. Consider $\widetilde \phi(s) = \phi(s^{-1})$. Then $\|\widetilde \phi\|_{\MCB(G)}=\|\phi\|_{\MCB(G)}$. Take $\varphi \in A(G) \to x_\varphi \in L^1(\cL G)$ the isometry described in Lemma \ref{lemma=isom_AG_L1}. Define $T_\phi^1$ by $T_\phi^1(x_\varphi)= x_{\widetilde \phi \varphi}$. It is a map of norm at most $\|\phi\|_{\MCB(G)}$, and by Lemma  \ref{lemma=isom_AG_L1} it satisfies $P_F T_\phi^1(x) P_F = M_\phi (P_F x P_F)$.

The general case follows by interpolation. We claim that the maps $T_\phi^1$ and $T_\phi^\infty$ are compatible with respect to the pair $(A_0,A_1)$ given by Theorem \ref{thm=interpolation_BB}. Indeed, let $a \in A_0 \cap A_1$. We have to show that $j_1^{-1}(T_\phi^1(j_1(a)) = j_\infty^{-1}(T_\phi^\infty(j_\infty(a))$. By \eqref{item=semi_injective} it suffices to have that $u_F(j_1^{-1}(T_\phi^1(j_1(a))) = u_F (j_\infty^{-1}(T_\phi^\infty(j_\infty(a)))$ for every relatively compact $F \subseteq G$. But by \eqref{item=diagram}, this equality is equivalent to $j_{1,F}^{-1}(M_\phi(j_{1,F}(u_F(a))) = j_{\infty,F}^{-1}(M_\phi(j_{\infty,F}(u_F(a)))$, which holds by \eqref{item=jpF}.

 By interpolation, there is therefore a map $T_\phi^p:L^p(\cL G) \to L^p(\cL G)$ of norm at most $\|\phi\|_{\MCB(G)}$. It satisfies $P_F T_\phi^p(x) P_F = M_\phi (P_F x P_F)$ by \eqref{item=diagram} and \eqref{item=jpF} again. 

The fact that $T_\phi^p$ is completely bounded follows from the argument in \cite[Theorem 1.6]{CanHaa} (see the proof of Theorem \ref{thm=embedding_in_ultraproduct} for details): apply the preceding to the function $\phi(s,k) = \phi(s)$ on $G \times SU(2)$ and use that $\cL SU(2) \simeq \oplus_{n \geq 1} M_n$.
\end{proof}

\section{Transference}\label{Sect=transference}

This section is devoted to the equality $\geq$ of Question \ref{Que-One}. The proof relies on transference techniques \cite[Lemma 2.4]{NeuRic}, which we adapt to non-discrete groups. 

%% \begin{dfn}
%% The support of an element $x$ of $L_p(\cL G)$ is the closed subset of $G$ defined by $s \notin support(x)$ if $s$ has a neighbourhood $V$ such that $P_F x P_{F'}=0$ for every compact $F,F'$ satisfying $F {F'}^{-1} \subset V$.
%% \end{dfn}
%% Example 1~: If $f \in C_C(G)$, the support of $\lambda(f) \in \cL G$ is the support of $f$.

%% Example 2~: If $x \in L_2(\cL G)$ corresponds to $f \in L^2(G)$, the support of $x$ is the essential support of $f$.

%% \begin{dfn}
%% An element $x$ of $L_p(\cL G)$ is said lattice-positive if for every compact $F$, $P_F x P_F \xi$ is nonnegative for every positive $\xi \in L^2(G)$.
%% \end{dfn}
%% Example 1~: If $f \in C_C(G)$, $\lambda(f) \in \cL G$ is lattice-positive if and only if $f(s)\geq 0$ for every $s \in G$.

%% Example 2~: If $x \in L_2(\cL G)$ corresponds to $f \in L^2(G)$, $x$ is lattice-positive if and only if $f(s)\geq 0$ for almost every $s \in G$.

%% \begin{lem}\label{lem=control_of_support}
%% If $x \in L_p(\cL G)$, the support of $x^*$ is $support(x)^{-1}$.

%% If $1\leq p,q,r \leq \infty$ with $1/p+1/q=1/r$, and if $x \in L_p(\cL G)$ and $y \in L_q(\cL G)$, then the support of $xy \in L_r(\cL G)$ is contained in $support(x)support(y)$. If $x,y$ are lattice-positive, so are $x^*$ and $xy$
%% \end{lem}
%% \begin{proof}
%% Straightforward computation.
%% \end{proof}

\begin{lem}\label{lem=recover_phi_from_Iphi} Let $G$ be a locally compact group and let $1 < p < \infty$ be such that the conjugate exponent $q$ satisfies $p/q \in \Q$. Then there is a net $x_\alpha \in L^p(\cL G)$ and $y_\alpha \in L^q(\cL G)$ such that $\|x_\alpha\|_p = \|y_\alpha\|_{q} = 1$ and for all $\phi \in \MCB(G)$, 
\[ \lim_\alpha \langle T_\phi^p x_\alpha,y_\alpha^*\rangle = \phi(e).\]
\end{lem}
\begin{proof}
The proof relies on Theorem \ref{thm=interpolation_BB}. Since $p/q \in \Q$ there exists $r<\infty$ and $k,k' \in \N$ such that $\frac 1 p = \frac k r$ and $\frac 1{q} = \frac{k'}{r}$.

Take a net of $\R^+$-valued fuctions $h_\alpha\in A_c(G)$ such that the support $S_\alpha$ of $h_\alpha$ converge to $\{e\}$. This means that every neighbourhood of $e$ in $G$ contains $S_\alpha$ for all $n$ large enough. By \eqref{item=intersection_big} in Theorem \ref{thm=interpolation_BB}, $j_\infty^{-1}(\lambda(h_\alpha))$ belongs to $A_0 \cap A_1$, we can therefore define $a_\alpha =j_{2r}(j_\infty^{-1}(\lambda(h_\alpha))))\in L_{2r}(\cL G)$. Normalise $h_\alpha$ so that $\|a_\alpha\|_{2r}=1$. Consider the unit vectors $x_\alpha \in L_p(\cL G)$ and $y_\alpha \in L_q(\cL G)$ defined by $x_\alpha = |a_\alpha|^{2k}$ and $y_\alpha = |a_\alpha|^{2k'}$. By \eqref{item=intersection_big2} in Theorem \ref{thm=interpolation_BB}, we can write $x_\alpha = j_p( j_\infty^{-1}(\lambda(f_\alpha)))$ for a non-negative function $f_\alpha \in A_c(G)$ with support contained in $(S_\alpha^{-1}S_\alpha)^k$. Similarly $j_q^{-1}(y_\alpha)$ belongs to $A_0 \cap A_1$, and $(j_1(j_q^{-1}(y_\alpha)))^* \in L_1(\cL G)$ corresponds to  $g_\alpha \in A(G)$ satsifying $g_\alpha(s) \geq 0$ for every $s \in S$.

For every $\phi \in \MCB(G)$, since the maps $T_\phi^p$ are compatible with respect to the interpolation given by Theorem \ref{thm=interpolation_BB}, we can write by \eqref{eq=duality_compatibility},
\[ 
\langle T_\phi^p x_\alpha,y_\alpha^*\rangle_{L^p, L^q} = \langle T_\phi^\infty(\lambda(f_\alpha)),(j_1(j_q^{-1}(y_\alpha)))^*\rangle_{L^\infty, L^1} = \int_G \phi(s) f_\alpha(s) g_\alpha(s) ds.
\]
Taking $\phi=1$ we get $\int f_\alpha(s) g_\alpha(s) ds = \langle x_\alpha,y_\alpha\rangle = \|a_\alpha\|_{2r}^{2r} = 1$. Hence,
\begin{eqnarray*} 
|\langle T_\phi^p x_\alpha,y_\alpha^*\rangle - \phi(e)| &=& |\int (\phi(s)-\phi(e)) f_\alpha(s) g_\alpha(s) ds|\\ & \leq & \sup_{s \in \textrm{support}(f_\alpha)} |\phi(s) - \phi(e)|.
\end{eqnarray*}
But $\phi$ is continuous and the support of $f_\alpha$ converges to $\{e\}$. This proves the lemma.
\end{proof}

\begin{thm}\label{Thm-TransferenceEstimate}
Let $G$ be a locally compact group and let $1 \leq p \leq \infty$.  Let $\phi \in \MCB(G)$. Then,
\[
\Vert M_{\phi}^p \Vert_{\CB(\cS_p(L^2G))} \leq \Vert T_{\phi}^p \Vert_{\CB(L^p(\cL G))}.
\]
\end{thm}
\begin{proof}

Let $q$ be the conjugate exponent of $p$. The values of $p$ for which $p/q \in \Q$ form a dense subset of $[1,\infty]$. It suffices to prove the theorem for such a $p$. Indeed, since $M_\phi^p$ is defined by means of interpolation, it follows from \cite[Theorem 4.1.2 and Section 2.4 (6)]{BerghLof} and the re-iteration theorem that the logarithm of $\Vert M_{\phi}^p \Vert_{\CB(\cS_p(L^2G))}$ is a convex, hence continuous function in $p$. Similarly, $\Vert T_{\phi}^p \Vert_{\CB(\cS_p(L^2G))}$  is continuous in $p$. 

Let $F = \{ s_1, \ldots, s_k \} \subseteq G$ be a finite subset, and consider two matrices $a,b \in M_k(\C)$. We claim that
\[ |\sum_{i,j} \phi(s_i s_j^{-1}) a_{i,j}  b_{j,i}| \leq \Vert T_{\phi}^p \Vert_{\CB(L^p(\cL G))} \|a\|_{\cS_p^k} \|b\|_{\cS_q^k}.\]
By Theorem \ref{thm=mult_de_symb_continu} this implies that $\Vert M_{\phi}^p \Vert_{B(\cS_p(L^2G))} \leq \Vert T^p_\phi  \Vert_{\CB(L^p(\cL G)) }$. The inequality $\Vert M_\phi^{p} \Vert_{\CB(\cS_p(L^2G))} \leq \Vert T^p_\phi  \Vert_{\CB(L^p(\cL G)) }$ follows similarly, by taking $a_{i,j} \in \cS_p^n$, $b_{i,j} \in \cS_q^n$ instead of $\mathbb{C}$.

Take $x_\alpha,y_\alpha$ given by Lemma \ref{lem=recover_phi_from_Iphi}, and consider the element $A_\alpha \in \cS_p^k \otimes L^p(\cL G)$ and $B_\alpha \in \cS_q^k \otimes L^q(\cL G)$ defined by
\[A_\alpha = (a_{i,j} \lambda(s_i) x_\alpha \lambda(s_j^{-1}))_{i,j \leq k}, B_\alpha = (b_{i,j} \lambda(s_i) y_\alpha \lambda(s_j^{-1}))_{i,j \leq k}.\]
$A_\alpha$ is obtained from $a \otimes x_\alpha$ by conjugating by the unitary $\sum_i e_{i,i} \otimes \lambda(s_i)$, so that $\|A_\alpha\|_{p} = \|a \otimes x_\alpha\|_p = \|a\|_{\cS_p^k}$. Similarly $\|B_\alpha\|_q = \|b\|_{\cS_q^k}$. Notice now that 
\[ \lambda(s_i^{-1}) T_\phi^p(\lambda(s_i) x_\alpha \lambda(s_j)^{-1}) \lambda(s_j) = T_{\phi_{i,j}}^p(x_\alpha)\]
where $\phi_{i,j}(s) = \phi(s_i s s_j^{-1})$. In particular $\phi_{i,j} \in \MCB(G)$ and $\phi_{i,j}(e) = \phi(s_i s_j^{-1})$. Lemma \ref{lem=recover_phi_from_Iphi} gives 
\begin{eqnarray*} \sum_{i,j} \phi(s_i s_j^{-1}) a_{i,j}  b_{j,i} & = & \lim_\alpha \sum_{i,j} a_{i,j}  b_{j,i} \langle T_{\phi_{i,j}}^p(x_\alpha),y_\alpha \rangle \\ & = & \lim_\alpha 
\langle (id \otimes T_\phi^p)(A_\alpha),B_\alpha\rangle.\end{eqnarray*}
The claim follows from the inequality \begin{eqnarray*}|\langle (id \otimes T_\phi^p)(A_\alpha),B_\alpha\rangle| &\leq& \Vert T_{\phi}^p \Vert_{\CB(L^p(\cL G))} \|A_\alpha\|_{p} \|B_\alpha\|_{q} \\&=& \Vert T_{\phi}^p \Vert_{\CB(L^p(\cL G))} \|a\|_{\cS_p^k} \|b\|_{\cS_q^k},\end{eqnarray*} valid for every $n$.
\end{proof}

%% \begin{prop}
%% Let $(X,\mu)$ be a measure space, and $D \in L_\infty(X,\mu)$ a real valued function such that $\inf_{x} D(x)>0$. Then the map $i_1:x \in \cS_1(L^2(X)) \mapsto D^{\alpha} x D^{1-\alpha} \in  B(L^2(X))$ makes $(A_0:=B(L^2 X),A_1:=i(\cS_1(L^2(X))))$ a compatible pair such that the map $i_p:x \in \cS_p(L^2 X) \mapsto D^{\theta/p} x D^{(1-\theta)/p}$ is an isometry onto $A_{1/p}$.
%% \end{prop}
%% We will call such a compatible couple a special couple (HORRIBLE TERMINOLOGY).

\section{Amenable groups}\label{Sect=amenable}

We prove that the answer to Question \ref{Que-One} is affirmative for amenable locally compact groups. This generalizes  beyond (amenable) discrete groups the result by Neuwirth and Ricard \cite{NeuRic}. 

\begin{thm}\label{thm=contraction_well_defined}
Let $F \subset G$ be a relatively compact Borel subset with positive measure, and $P_F:L_2(G) \to L_2(F)$ the orthogonal projection. Let $1 \leq p \leq \infty$. For any $x \in L^p(\cL G)$, $P_F x P_F \in \cS_p(L_2 G)$ and $\| P_F x P_F\|_{\cS_p(L^2F)} \leq |F|^{1/p} \|x\|_{L^p(\cL G)}$.
\end{thm}
\begin{proof}
When $p=\infty$ this is obvious. 

Assume $p=1$, and take $x \in L_1(\cL G)$. Let $\phi \in A(G)$ corresponding to $x$, and write $\phi(s) =\langle \lambda(s) \xi,\eta\rangle$ for $\xi,\eta \in L^2(G)$ with $\|\xi\| \|\eta\| = \|x\|_{L^1(\cL G)}$. By Lemma \ref{lemma=isom_AG_L1}, $P_F x P_F = ( \langle \lambda(s^{-1})\xi,\lambda(t^{-1}) \eta\rangle)_{s,t \in F} = A B^*$ where $A = (\xi(s t))_{s \in F,t\in G}$ and $B = (\eta(st))_{s \in F, t \in G}$. But $\|A\|_{\cS_2(L^2G)}^2 = \iint |\xi(st)|^2 1_{s \in F} dt ds = |F| \|\xi\|^2$, and similarly $\|B\|_{\cS_2(L^2G)} = |F|\|\eta\|^2$, so that $\| P_F x P_F\|_{\cS_1(L^2G)} \leq \|A\|_{\cS_2(L^2G)} \|B\|_{\cS_2(L^2G)} =|F| \|x\|_{L^1(\cL G)}$. This proves the case $p=1$. 

The general case follows by the interpolation Theorem \ref{thm=interpolation_BB}: the cases $p=1$ and $p=\infty$ above say that $u_F$ maps $A_0 \to A_{0,F}$ with norm $1$ and $A_1 \to A_{1,F}$ with norm $|F|$. Hence it maps $A_{1/p} \to A_{1/p,F}$ with norm less than $|F|^{1/p}$. By \eqref{item=diagram} this is exactly saying that $P_F x P_F \in \cS_p(L^2F)$ and $\|P_F x P_F \|_{\cS_p(L^2F)} \leq |F|^{1/p} \|x\|_{L_p(\cL G)}$.
\end{proof}

\begin{thm}\label{thm=embedding_in_ultraproduct} Let $G$ be an amenable locally compact group and $1< p <\infty$. Then there is an ultrafilter $\cU$ on some set and a completely isometric embedding $i:L^p(\cL G) \to \prod_{\cU} \cS_p(L^2G)$ that intertwines Fourier and Schur multipliers.
\end{thm}
\begin{proof}
By amenability there exists a  F\o{}lner net $(F_\alpha)_\alpha$~: $F_\alpha$ are compact subsets with positive measure such that $\lim_\alpha |F_\alpha \cap g F_\alpha|/|F_\alpha| =1$ for all $g \in G$. Choose an ultrafilter $\cU$ refining the net $\alpha$, denote by $P_\alpha$ the orthogonal projection on $L^2(F_\alpha)$ and consider the map $i_p:L^p(\cL G) \to \prod_{\cU} \cS_p(L^2 F_\alpha)$ defined by $i_p(x) = (P_\alpha x P_\alpha/|F_\alpha|^{1/p})_\alpha$. By Theorem \ref{thm=contraction_well_defined}, $i_p$ is a well defined contraction, and it intertwines Fourier and Schur multipliers by definition of Fourier multipliers. We use a duality argument to prove that it is isometric. To do this take $q$ the conjugate exponent of $p$ and consider $x \in L^p(\cL G)$ and $y \in L^q(\cL G)$. We claim that
\begin{equation}\label{eq=ipiq} \langle i_p(x),i_q(y)^*\rangle = \langle x,y^*\rangle.\end{equation}
where on the left-side, the pairing is given by $\langle (x_\alpha)_\alpha,(y_\alpha) \rangle =\lim_{\cU} \Tr(x_\alpha y_\alpha)$ and on the right-side the pairing is given by $\langle x,y\rangle = \int xy d\psi$. Since $L^q$ is isometrically the dual of $L^p$ for this pairing, this clearly implies that $i_p$ (and $i_q$) are isometries.

Let us first prove \eqref{eq=ipiq} in the particular case that $p=\infty,q=1$, $x = \lambda(f)$ with $f \in L^1(G)$ and $y^*$ corresponds to $\varphi \in  A(G)$. We compute $\langle  i_\infty(x), i_1(y)^* \rangle$, using Lemma \ref{lemma=isom_AG_L1}, \eqref{eq=kernel_lambdaf} and then \eqref{eq=formula_for_trace},
\[
\begin{split}
&  \langle  i_\infty(x), i_1(y)^* \rangle\\
 = &
\lim_{\alpha,\cU} \frac{1}{\vert F_\alpha \vert} \Tr\left( P_\alpha \lambda(f) P_\alpha y^* P_\alpha \right)\\
= &  \lim_{\alpha,\cU} \frac{1}{\vert F_\alpha \vert} \Tr\left( (f(st^{-1}) \Delta(t^{-1})_{s,t \in F_\alpha} (\varphi(ts^{-1}))_{s,t \in F_\alpha} \right)\\
= &  \lim_{\alpha,\cU} \frac{1}{\vert F_\alpha \vert} \int_{F_\alpha} \int_{F_\alpha} \Delta(t^{-1}) f(st^{-1}) \varphi(st^{-1}) ds dt \\
 = & \lim_{\alpha,\cU} \frac{1}{\vert F_\alpha \vert} \int_{G} f(u) \varphi(u)(\int_{F_\alpha} \chi_{F_\alpha}(u^{-1}s) ds) du \\
= & \lim_{\alpha,\cU}\int_{G} f(u) \varphi(u) \frac{\vert F_\alpha \cap uF_\alpha\vert}{\vert F_\alpha \vert} du \\
= & \int_G f(u) \varphi(u) du =\langle x,y^*\rangle.
\end{split}
\]
On the fifth line we made the change of variable $u=st^{-1}$, and the last line is justified by the dominated convergence theorem. 

We now use the interpolation Theorem \ref{thm=interpolation_BB} and its notation. We claim that $\langle i_p(j_p(a)) ,i_q(j_q(b))^*\rangle = \langle j_p(a),j_q(b)^*\rangle$ holds for all $p,q$ with $1/p+1/q=1$ and all $a,b \in A_0 \cap A_1$. This will conclude the proof of \eqref{eq=ipiq} since by general interpolation theory, $j_p(A_0 \cap A_1)$ is dense in $L^p(\cL G)$ for every $1<p<\infty$. By the interpolation equation \eqref{eq=duality_compatibility} and \eqref{item=diagram} in Theorem \ref{thm=interpolation_BB} neither $\langle i_p(j_p(a)) ,i_q(j_q(b))^*\rangle$ nor $\langle j_p(a),j_q(b)^*\rangle$ depend on $p$, and we know (case $p=\infty$) that they are equal when $a \in E$.   These quantities are therefore also (case $p=1$) equal when $b \in E$, and by the norm density of $j_1(E)$ in $L^1(\cL G)$, they are equal for every $b$. This proves \eqref{eq=ipiq}.

To prove that $i_p$ is completely isometric, the same proof can be applied. Otherwise we can use a classical argument \cite[Theorem 1.6]{CanHaa} and consider $K=SU(2)$. This is very convenient because as is well-known $\cL K = \oplus_{n \geq 1} M_n$, but all we need is that $K$ is a compact group with irreducible unitary representations of arbitrarily large dimension. Remark that $G \times K$ is amenable and that $F_\alpha \times K$ is a F\o lner net. By what we just proved, $i_p \otimes id: L^p( \cL (G \times K)) \to \prod_{\cU} \cS_p(L^2(G \times K))$ is isometric. But since $\cL (G \times K) = \cL G   \otimes \cL K \simeq \oplus_n M_n(\cL G)$, this means that
\[ id \otimes i_p: \ell^p\{ \cS_p^n \otimes L^p(\cL G), n \geq 1\} \to \prod_{\cU} \ell^p\{ \cS_p^n \otimes \cS_p(L^2 G), n\geq 1\}\]
is isometric, which implies that $i_p$ was completely isometric.
\end{proof}
As a straightforward consequence.
\begin{corollary}\label{Cor-MainEquality}
Let $G$ be an amenable group and let $1 \leq p \leq \infty$. Let $\phi \in \MCB(G)$. Then,
\[
\Vert T_\phi^p \Vert_{B(L^p(\cL G))} \leq \Vert M_\phi^p \Vert_{B(\cS_p(L^2 G))},
\]
\[
\Vert T_\phi^p \Vert_{\CB(L^p(\cL G))} \leq \Vert M_\phi^p \Vert_{\CB(\cS_p(L^2 G))}.
\]
\end{corollary}

\section{Interpolation}\label{Sect=interpolation}
%Let $(X,\mu)$ be a measure space. A map $u:B(L^2(X,\mu)) \to B(L^2(X,\mu))$ is called ??? if there exist positive invertible elements $L,R \in L_\infty(X,\mu)$ such that $u(x) =LxR$.

In this section, we collect the necessary results from \cite{TerpII} and \cite{Hilsum} in order to interpret non-commutative $L^p$-spaces as interpolation spaces. This is done in Theorem \ref{thm=interpolation_BB}. Using this theorem, we are able to prove our results in a self-contained way. The theorem is technical in nature and its conceptual consequences are stated in the earlier sections. 

Let $\cM$ be a von Neumann algebra, and consider $L_p(\cM)$ ($1 \leq p \leq \infty$) the associated non-commutative $L_p$-spaces together with a bilinear duality bracket $\langle \cdot,\cdot\rangle_{L_p,L_q}$ for $1/p+1/q=1$.  For simplicity we will only use this notion in two particular explicit cases: the first is when $\cM = B(\cH)$ for a Hilbert space $\cH$, and in this case $L_p(\cM)$ is $\cS_p(\cH)$ and the duality $\langle A,B\rangle_{L^p, L^q} = \Tr(AB)$. The second case is when $\cM'$ is equipped with a normal faithful weight $\psi$, $L^p(\cM)$ is the associated Connes-Hilsum space, and the duality is $\langle a,b\rangle = \int ab d\psi = \int ba d\psi$. The special case $\cM = \cL G$ is recalled in Section \ref{subsect=Connes-Hilsum}. The special case $\cM = B(\cH)$ and $\psi$ is the canonical trace on $B(\cH)'\simeq \C$ will then yield $L^p(\cM) = \cS_p(\cH)$.
 
\begin{dfn} Let $\cM$, $L^p(\cM)$ be as above. An \emph{interpolation scale for $\{(L_p(\cM), 1\leq p \leq \infty\}$} is a compatible pair $\overline A= (A_0,A_1)$ and a family of isometric isomorphisms $j_p:A_{1/p} \to L_p(\cM)$ such that for every $a ,b \in A_0 \cap A_1$ and every $1 \leq p,q \leq \infty$ satisfying $1/p+1/q=1$, 
\begin{equation}\label{eq=duality_compatibility} \langle j_p(a),j_q(b)^*\rangle_{L_p,L_q} = \langle j_1(a),j_\infty(b)^*\rangle_{L_1,L_\infty}.
\end{equation} 
\end{dfn}
In this definition $A_\theta \subset A_0 + A_1$ stands for the complex interpolation space between $A_0$ and $A_1$ with parameter $\theta$. We write $\Sigma(\overline{A}) =  A_0 + A_1$. 

%% \begin{rem}\label{rem=j_p_omitted}When there is no risk of confusion we might omit the maps $j_p$, identify $L_p(\cM)$ with a subspace of $L_1(\cM) + L_\infty(\cM)$ and just write
%% \[ \langle a,b^*\rangle_{L_p,L_q} = \langle a,b^*\rangle_{L_1,L_\infty} \forall a,b \in L_1(\cM) \cap \cM.\]
%% \end{rem}
%In general there are lots of different interpolation scales for $\{(L_p(\cM), 1\leq p \leq \infty\}$. For example for $\cM = B(H)$, for every positive invertible operators $D_1,D_2 \in B(H)$, the maps $u_p:A \in \cS_p(H) \mapsto D_1^{1/p} A D_2^{1/p} \in B(H)$ in particular define compatible pair on $(\cS_1(H),B(H))$ and $u_p$ is then an isometric isomorphism from $\cS_p(H))$ onto $[u_1(\cS_1(H)),B(H)]_{1/p}$. The duality condition is immediate.

\begin{thm}\label{thm=interpolation_BB}
Let $F \subset G$ be a relatively compact Borel subset.
There exist interpolation scales $\{\overline A= (A_0,A_1), j_p:A_{1/p} \to L_p(\cL G)\}$, and $\{\overline A_F= (A_{0,F},A_{1,F}), j_{p,F} :A_{1/p,F} \to \cS_p(L^2(F))\}$ with $A_{1,F} \subset A_{0,F}$ and a bounded map $u_F: \Sigma(\overline A) \to \Sigma(\overline A_F) = A_{0,F}$ such that
\begin{enumerate}[(i)]
%\item \label{item=jp} for all $1 \leq p \leq \infty$ there is an isometric isomorphism $j_p: A_{1/p} \to L_p(\cL G)$.
\item \label{item=jpF}for all $1 \leq p \leq \infty$, $j_{p,F}$ extends to a continuous isomorphism $j_{p,F}:A_{0,F} \to B(L^2 F)$,  and $j_{p,F} \circ j_{q,F}^{-1}$ commutes with Schur multipliers for every $p,q$. 
\item \label{item=diagram}The following diagram commutes.
\begin{equation}\label{Diag-BlackBox}
 \xymatrix{
%x \ar@{|->}[d]  \in 
L_p(\cL G) \ar[d]^{x \mapsto P_F x P_F} & A_{1/p} \ar[l]^{j_p} \ar[d]^{u_F}\\
%P_F x P_F  \in 
B(L^2 F)  & A_{0,F}\ar[l]^{j_{p,F}} }
\end{equation}
    
%\item \label{item=duality_jp} For every $a,b \in A_0 \cap A_1$, and $1/p+1/q=1$,
%\[\langle j_p(a),j_q(b)^*\rangle_{L_p,L_q} = \langle j_1(a),j_\infty(b)^*\rangle_{L_1,L_\infty}.\]
%\item \label{item=duality_jpF} For every $a,b \in (A_{0,F} \cap A_{1,F}) = A_{1,F}$, and $1/p+1/q=1$,
%\[\langle j_{p,F}(a),j_{q,F}(b)^*\rangle_{\cS_p,\cS_q} = \langle j_{2,F}(a),j_{2,F}(b)^*\rangle_{\cS_2,\cS_2}.\]
\item \label{item=intersection_big} The space $E = \textrm{span}\{ j_\infty^{-1}(\lambda(f)^\ast \lambda(g)), f,g \in C_c(G)\}$ is contained in $A_0 \cap A_1$. Furthermore, $\{j_1(a),a\in E\}$ is norm dense in $L^1(\cL G)$.
\item \label{item=semi_injective} Let $a \in \Sigma(\overline{A})$. If $u_F(a) = 0$ for every $F \subseteq G$ relatively compact, then $a = 0$.
\item \label{item=intersection_big2} For every $a = j_\infty^{-1}(\lambda(f_a)),b =j_\infty^{-1}(\lambda(f_b))\in E$, every $1\leq p,q,r \leq \infty$ such that $1/p+1/q=1/r$, there exist $c=j_\infty^{-1}(\lambda(f_c)),d = j_\infty^{-1}(\lambda(f_d))\in E$ such that $j_p(a)^* = j_p(c)$ and $j_p(a) j_q(b) = j_r(d)$. Moreover, $support(f_c) = support(f_a)^{-1}$, $support(f_d) \subset support(f_a) support(f_b)$, and if $f_a$ and $f_b$ are non-negative so are $f_c$ and $f_d$. If $f_a$ is non-negative, $j_1(a) \in L^1(\cL G)$ corresponds to a non-negative (i.e. point-wise non-negative) element in $A(G)$.
\end{enumerate}
\end{thm}

 The proof will use Terp's construction \cite{TerpII} of non-commutative $L^p$-spaces as interpolation spaces, that we now recall.

The construction starts from a triple $(\cM, \varphi, \psi)$. Here, $\cM$ is a von Neumann algebra.   $\varphi$ is  a normal, semi-finite, faithful weight on $\cM$ and $\psi$ is  a normal, semi-finite, faithful weight on the commutant $\cM '$.   Associated with $\varphi$, we have a GNS-triple $(\cH, \pi, \Lambda)$, a modular conjugation $J$, modular operator $\Delta$ and modular automorphism group $\sigma$ (see \cite[p 92]{TakII}). We identify $\cM$ with $\pi(\cM)$ and omit $\pi$ in the notation. 

We turn  $(\cM, \cM_\ast)$   into a compatible couple of Banach spaces as follows. Let $\cM \cap \cM_\ast$ be the set of $x \in \cM$, for which there exists a $\varphi_x \in \cM_\ast$ such that
\[
\forall y,z \in \nphi: \langle \varphi_x, z^\ast y \rangle = \langle J x^\ast J \Lambda(y), \Lambda(z) \rangle. 
\]
For $x \in \cM \cap \cM_\ast$, we set the norm $\Vert x \Vert_{\cM \cap \cM_\ast} = \max \{ \Vert x \Vert_{\cM}, \Vert \varphi_x \Vert_{\cM_\ast} \}$, in which it becomes a Banach space. Naturally, we find two embeddings $i_\infty: \cM \cap \cM_\ast \rightarrow \cM: x \mapsto x$ and $i_1: \cM \cap \cM_\ast: x \mapsto \varphi_x$.  Dualizing the embeddings (and restricting to $\cM_\ast$) we obtain a commutative diagram:
 \begin{equation} 
 \xymatrix{
 & \cM_\ast\ar@{^{(}->}[dr]^{i_\infty^{\ast}}  & \\
   \cM \cap \cM_\ast  \ar@{^{(}->}[ur]^{i_1}\ar@{^{(}->}[dr]_{i_\infty}  & &( \cM \cap \cM_\ast )^\ast, \\
& \cM  \ar@{^{(}->}[ur]^{i_1^\ast} &}    \label{EqnLpIzu}
\end{equation}
 which turns $(\cM, \cM_\ast)$ into a compatible couple of Banach spaces. Moreover, $i_\infty^* i_1(\cM \cap \cM_\ast)$ is the intersection of $i_\infty^\ast(\cM_\ast)$ and $i_1^\ast(\cM)$, so that our notation is justified. We have $\mphi \subseteq \cM \cap \cM_\ast$. 

  We write $L^p(\cM)$ for the Connes-Hilsum $L^p$-space \cite{Hilsum}. We let $d = d\varphi/d\psi$ be the spatial derivative. For any $x \in \nphi$, we have 
\[
[xd^{\frac{1}{2p}}] \in L^{2p}(\cM) \textrm{ and } d^{\frac{1}{2p}} x^* \in L^{2p}(\cM)\]
c.f. \cite[Theorem 26]{TerpII}. The main result of \cite{TerpII} states that for $1 \leq p \leq \infty$, there exists a unique isometry $j_p: (\cM, \cM_\ast)_{1/p} \rightarrow L^p(\cM)$ such that for $x,y \in \nphi$, we have
\begin{equation}\label{Eqn-LpIso}
j_p(x^\ast y ) =  d^{\frac{1}{2p}} x^\ast  \cdot [y d^{\frac{1}{2p}}] \in L^{2p}(\cM) \cdot L^{2p}(\cM) \subseteq L^p(\cM).
\end{equation}
Moreover, for $1 \leq p < \infty$, $j_p$ is an isomorphism ($j_\infty$ is just the identity map from $\cM$ to $\cM$). The following proposition is exactly \cite[Eqn. (56)]{TerpII}.

\begin{prop}\label{Prop-InterpolationScale}
Let $(\cM, \varphi, \psi)$ be a triple as above.  $\{  (\cM, \cM_\ast), j_p: (\cM, \cM_\ast)_{1/p} \rightarrow L^p(\cM) \}$ is an interpolation scale. 
\end{prop}
\begin{rem}
In \cite{Kosaki, Izumi} it is shown that depending on a {\it complex interpolation parameter} $z \in \mathbb{C}$, the triple $(\cM, \varphi, \psi)$ gives a compatible couple $(\cM, \cM_\ast)$ by taking different embeddings of (subsets of) $\cM$ into $\cM_\ast$. These all give rise to an interpolation scale, the special case treated here being $z=0$. Since the Schur and Fourier multipliers in this paper commute with $\sigma$, there is in fact no preference for the choice of the interpolation parameter. 
\end{rem}

\begin{rem}
In the next proof we  use the following standard fact  
\cite[Chapter $\textrm{VII}.3$]{TakII} on the Plancherel weight $\varphi$ on $\cL G$. For $\xi \in L^1 G \cap L^2 G$ we have $\lambda(\xi) \in \nphi$ and the mapping $\lambda(L^1 G \cap L^2 G) \rightarrow L^2 G: \lambda(\xi) \mapsto \xi$ is a $\sigma$-weak/norm core for the GNS-map of $\varphi$. 
\end{rem}

\begin{proof}[Proof of Theorem \ref{thm=interpolation_BB}]
Consider the triple $(\cL G, \varphi, \psi)$, see Sections \ref{subsect=groupvna} and \ref{subsect=Connes-Hilsum} for notation. We will describe precisely the GNS-triple and and modular conjugation $J$ in the proof of (\ref{item=semi_injective}). Terp's construction gives a compatible couple $\overline{A} = (A_0,A_1)$ and an interpolation scale  $\{ \overline{A}, j_p: A_{1/p} \rightarrow L^p(\cL G)\}$ satisfying \eqref{Eqn-LpIso}. 

Consider also $(B(L^2F), \Tr_{\Delta}, \Tr ')$. Here, 
\[
\Tr_{\Delta}(x) = \Tr(\Delta_F^{\frac{1}{2}}  x \Delta_F^{\frac{1}{2}} ), \qquad x \in B(L^2 F)^+,
\]
where $\Delta_F$  is the bounded operator given by $\Delta$ restricted to $L^2 F$.  Furthermore, $\Tr '$ is unique state on $B(L^2F)' = \C$. In this case, $L^p(B(L^2F)) = \cS_p(L^2 F)$  and we have spatial derivative $d \Tr_\Delta / d\Tr ' = \Delta_F$.  We {\it define} the map, 
\[
j_{p,F}: B(L^2 F) \rightarrow B(L^2 F): x \mapsto \Delta_F^{\frac{1}{2p}}  x \Delta_F^{\frac{1}{2p}} . 
\]
Then, the restriction of $j_{p,F}$ to $A_{1/p,F}$ is an isometric isomorphism onto $\cS_p( L^2 F )$ (by Terp's construction). In particular, we find an interpolation scale $\{\overline A_F= (A_{0,F},A_{1,F}), j_{p,F} :A_{1/p,F} \to \cS_p(L^2(F))\}$. The map $j_{p,F} \circ j_{q,F}^{-1}$, being a Schur multiplier, commutes with Schur multipliers. In all, we conclude (\ref{item=jpF}). 

(\ref{item=diagram}) Take  $x,y \in \nphi$. Then, for all $1 \leq p \leq \infty$ we have from \eqref{Eqn-LpIso} 
\begin{equation}\label{eq=uF_compatible}
j_{p,F}^{-1} (P_F j_p(j_\infty^{-1}(x^\ast y)) P_F) =  
\Delta_F^{-\frac{1}{2p}}  (P_F \Delta^{\frac{1}{2p}} x^\ast \cdot [ y \Delta_F^{\frac{1}{2p}}] P_F) \Delta_F^{-\frac{1}{2p}}   = P_F x^\ast y P_F.
\end{equation}
In the case $p=1$, \cite[Theorem 8]{TerpII} implies that the equation $j_{1,F}^{-1} (P_F j_1(z) P_F) = P_F j_\infty(z) P_F$ holds for all $z \in A_0 \cap A_1$. We can therefore define a bounded map $u_F: A_0 + A_1 \to A_{0,F} = B(L^2 F)$ by $u_F(x_0+x_1) = P_F j_\infty(x_0) P_F + j_{1,F}^{-1} (P_F j_1(x_1) P_F)$ if $x_0\ \in A_0, x_1 \in A_1$. By the definition of $u_F$, \eqref{item=diagram} then commutes for $p=1$ or $p=\infty$.

For $1< p < \infty$, \eqref{eq=uF_compatible} proves that $j_{p,F}^{-1} (P_F j_p(z) P_F) = u_F(z)$ for all $z \in j_\infty^{-1}(\mphi)$. 
Recall that $j_p(j_\infty^{-1}(\mphi))$ is dense in $L^p(\cL G)$ (by \cite[Theorem 26]{TerpII}, Eqn. (\ref{Eqn-LpIso}) and H\"older's inequality). This concludes (\ref{item=diagram}).
 
 (\ref{item=intersection_big})   For $\xi \in L^1G \cap L^2G$ we have $\lambda(\xi) \in \nphi$. Hence $\lambda(\xi)^\ast \lambda(\eta) \in \mphi \subseteq A_0 \cap A_1$ for $\xi,\eta \in C_c(G)$. Let $x \in \cL G$ be such that $\langle j_1 \circ j_\infty^{-1}(\lambda(\xi) \lambda(\eta)), x \rangle = 0$ for all $\xi, \eta \in C_c(G)$. Then, using the notation \eqref{eq=def_of_xistar} and \cite[Eqn. (38)]{TerpII} in the fourth equality,
\[
\begin{split}
0 =& \langle j_1 \circ j_\infty^{-1}(\lambda(\xi) \lambda(\eta)), x \rangle 
=  \langle \Delta^{\frac{1}{2}} \lambda(\xi) \cdot [\lambda(\eta) \Delta^{\frac{1}{2}} ], x \rangle \\
= & \int [\lambda(\xi \Delta^{\frac{1}{2}} )\Delta^{\frac{1}{2}}]  \cdot \Delta^{\frac{1}{2}} \lambda(\eta \Delta^{-\frac{1}{2}}) \cdot x d\psi 
=  \varphi(\lambda(\eta \Delta^{-\frac{1}{2}})x\lambda(\xi \Delta^{\frac{1}{2}})) \\
  = & \langle x \Delta^{\frac{1}{2}} \xi ,   \Delta^{\frac{1}{2}}  \eta^\ast \rangle 
\end{split}
\]
This implies that $x = 0$, since $C_c(G)$ is dense in $L^2G$ . Hence $j_1(E)$ is dense in $L^1(\cL G)$.

(\ref{item=semi_injective})  Let $a =a_0+a_1 \in A_0 + A_1$ be such that $u_F(a) = 0$ for all $F \subseteq G$ relatively compact. Write $j_\infty(a_0)=x \in \cL G$ and $j_1(a_1)=\omega_\phi \in \cL G_\ast$ corresponding to $\phi \in A(G)$ and $x_\phi \in L^1(\cL G)$. Firstly, let $y = \lambda (\xi), z = \lambda(\eta)$, with $\xi, \eta \in C_c(G)$. We will use the notation \eqref{eq=def_of_xistar}. We have $\Lambda(\lambda(\xi)) = \xi \in L^2 G$ and $J \Lambda(\lambda(\xi)) = \xi^\ast \Delta^{\frac{1}{2}} \in L^2 G$. Furthermore, $y^\ast z \in \mphi \subseteq j_\infty(A_0 \cap A_1)$. By Terp's construction, $\Sigma(\overline{A}) \subseteq (A_0 \cap A_1)^\ast$. We first claim that
\[
\langle a, j_\infty^{-1}(y^\ast z) \rangle_{(A_0 \cap A_1)^\ast, A_0 \cap A_1}  = 0.
\]
Indeed, taking $F$ that contains the supports of $\eta^\ast, \xi^\ast$ and $\xi^\ast \ast \eta$,
\begin{equation}\label{Eqn=earlycomputation}
\begin{split}
& a, j_\infty^{-1}(y^\ast z) \rangle_{(A_0 \cap A_1)^\ast, A_0 \cap A_1} \\
%= & \langle a_0, y^\ast z \rangle_{(A_0 \cap A_1)^\ast, A_0 \cap A_1} + 
%\langle \omega_\phi, y^\ast z \rangle_{(A_0 \cap A_1)^\ast, A_0 \cap A_1} \\
= &
\langle J x^\ast J \Lambda(z), \Lambda(y) \rangle + \langle \phi, y^\ast z \rangle_{A(G), \cL G} \\
= &
\langle x J \Lambda(y), J \Lambda(z) \rangle + \int_F \phi(s) (\xi^\ast \ast \eta)(s) ds \\
= &
\langle x \Delta^{\frac{1}{2}} \xi^\ast,  \Delta^{\frac{1}{2}} \eta^\ast \rangle + \int_F \phi(t s^{-1}) \overline{\xi(t^{-1})} \Delta(t^{-1}) \eta(s^{-1}) \Delta(s^{-1})  ds.
\end{split}
\end{equation}
We continue the equation using Lemma \ref{lemma=isom_AG_L1},
\[
\begin{split}
& \langle a, j_\infty^{-1}(y^\ast z) \rangle_{(A_0 \cap A_1)^\ast, A_0 \cap A_1} \\
= & \langle \Delta^{\frac{1}{2}} P_F x P_F \Delta^{\frac{1}{2}}  \xi^\ast  ,  \eta^\ast  \rangle + \langle P_F x_\phi P_F   \xi^\ast, \eta^\ast \rangle \\
= & \langle (P_F x P_F + \Delta^{-\frac{1}{2}}  P_F x_\phi P_F \Delta^{-\frac{1}{2}} )\Delta^{\frac{1}{2}}  \xi^\ast  ,  \Delta^{\frac{1}{2}}  \eta^\ast  \rangle \\
  = & \langle u_F(a) \Delta^{\frac{1}{2}}  \xi^\ast  ,  \Delta^{\frac{1}{2}}  \eta^\ast  \rangle =0.
\end{split}
\]
We claim that it follows that in fact 
\[
\langle a, w \rangle_{(A_0 \cap A_1)^\ast, A_0 \cap A_1} = 0,
\]
 for every $w \in A_0 \cap A_1$. By \cite[Theorem 8]{TerpII} it is enough to show this when $w \in \mphi$. By linearity we may assume that $w = y^\ast z$ with $y,z \in \nphi$. We thus need to show that
\begin{equation}\label{Eqn=finalstatement}
\langle J x^\ast J \Lambda(z), \Lambda(y) \rangle + \langle \phi, y^\ast z \rangle_{A(G), \cL G} = 0,
\end{equation}
c.f. the first three lines of \eqref{Eqn=earlycomputation}.  Since $\lambda(L^1 G \cap L^2 G)$ is a $\sigma$-weak/norm core for $\Lambda$, there are nets $\xi_i, \eta_i$ in $L^1 G \cap L^2 G$ such that $\lambda(\xi_i) \rightarrow y$ $\sigma$-weakly and $\Vert \xi_i - \Lambda(y) \Vert_{L^2 G} \rightarrow 0$ and similarly, $\lambda(\eta_i) \rightarrow z$ $\sigma$-weakly and $\Vert \eta_i - \Lambda(z) \Vert_{L^2 G} \rightarrow 0$. In fact, we may take $\xi_i, \eta_i \in C_c(G)$. Then, clearly \eqref{Eqn=finalstatement} follows from what we have proved.

(\ref{item=intersection_big2}) We claim that the required functions are given by:
\[
f_c(s) = \overline{f_a(s^{-1})} \Delta(s^{-1}), \qquad f_d(s) = (f_a \Delta^{-\frac{1}{2q}} ) \ast (f_b \Delta^{-\frac{1}{2p}}),
\]   
so that $c=a^\ast$ and $d = \sigma_{\frac{i}{2q}}(a)\sigma_{\frac{i}{2p}}(b)$. For notational convenience, suppose that $a = u^\ast v$, with $u,v \in \nphi$. In the general case one considers a linear combination of such a decomposition. Recall that  for any two unbounded operators $x$ and $y$ we have $(xy)^\ast \supseteq x^\ast y^\ast$. Then, using \eqref{Eqn-LpIso} and \cite[Theorem 4.(1)]{Hilsum} in the third equation,
\[
j_p( c ) = j_p(a^\ast)   = \Delta^{\frac{1}{2p}} v^\ast  \cdot [ u  \Delta^{\frac{1}{2p}} ]
=  \left(  \Delta^{\frac{1}{2p}} u^\ast  \cdot [ v  \Delta^{\frac{1}{2p}} ] \right)^\ast = j_p(a)^\ast.
\]
Furthermore, write still $a = u^\ast v$ and also $b = x^\ast y$ with $u,v, x,y \in \nphi$. We claim that we have the following  equalities,
\[
\begin{split}
& j_r(d) = j_r(  \sigma_{\frac{i}{2q}}(a)  \sigma_{\frac{i}{2p}}(b)   ) =   \Delta^{\frac{1}{2r}} \sigma_{\frac{i}{2q}}(a)   \cdot   [\sigma_{\frac{i}{2p}}(b)  \Delta^{\frac{1}{2r}} ] \\
 =\!\!\!\!\!{^?} \:\:\: &     
  \Delta^{\frac{1}{2p}} u^\ast \cdot  [v \Delta^{\frac{1}{2p}}] \cdot   \Delta^{\frac{1}{2q}} x^\ast \cdot  [y \Delta^{\frac{1}{2q}}]  
=    j_p(a) j_q(b).
\end{split}
\]
Only  $=\!\!\!\!\!{^?} \:\:\:$ needs justification, the other equalities are immediate from \eqref{Eqn-LpIso}. Firstly, using \cite[Lemma 22]{TerpII},
\[
\Delta^{\frac{1}{2r}} \sigma_{\frac{i}{2q}} (a) \supseteq \Delta^{\frac{1}{2p}} a \Delta^{\frac{1}{2q}} = \Delta^{\frac{1}{2p}} u^\ast v \Delta^{\frac{1}{2q}}. 
\]
Using that the left hand side is closed and taking the closure on the right hand side, $\Delta^{\frac{1}{2r}} \sigma_{\frac{i}{2q}} (a) \supseteq \Delta^{\frac{1}{2p}} u^\ast\cdot[ v \Delta^{\frac{1}{2q}} ] $. Since (use \eqref{Eqn-LpIso} for the right-hand side) both sides are in $L^{2r}(\cL G)$ we in fact have equality by \cite[Theorem 4.(1)]{Hilsum}.  Also, using the same argument in the second equality,
\[
[ \sigma_{\frac{i}{2p}}(b) \Delta^{\frac{1}{2r}} ]^\ast = \Delta^{\frac{1}{2r}} \sigma_{\frac{i}{2p}} (b) = \Delta^{\frac{1}{2q}} y^\ast \cdot [x \Delta^{\frac{1}{2p}}] = (\Delta^{\frac{1}{2p}} x^\ast \cdot [y \Delta^{\frac{1}{2q}}] )^\ast. 
\]
So, $[ \sigma_{\frac{i}{2p}}(b) \Delta^{\frac{1}{2r}} ] = \Delta^{\frac{1}{2p}} x^\ast \cdot [y \Delta^{\frac{1}{2q}}] $.
Furthermore,
\[
[v \Delta^{\frac{1}{2q}}] \cdot \Delta^{\frac{1}{2p}} x^\ast \supseteq v \cdot \Delta^{\frac{1}{2r}} x^\ast \subseteq [v \Delta^{\frac{1}{2q}}] \cdot \Delta^{\frac{1}{2p}} x^\ast. 
\]
Since each instance of this line is in $L^{2}(\cL G)$, the inclusions are equalities by \cite[Theorem 4.(1)]{Hilsum}. In all, we have proved that
\[
\begin{split}
 \Delta^{\frac{1}{2r}} \sigma_{\frac{i}{2q}}(a) \cdot   [\sigma_{\frac{i}{2p}}(b)  \Delta^{\frac{1}{2r}} ] = &
 \Delta^{\frac{1}{2p}} u^\ast\cdot[ v \Delta^{\frac{1}{2q}} ] \cdot \Delta^{\frac{1}{2p}} x^\ast \cdot [y \Delta^{\frac{1}{2q}}] \\ = &
\Delta^{\frac{1}{2p}} u^\ast \cdot  [v \Delta^{\frac{1}{2p}}] \cdot  \Delta^{\frac{1}{2q}} x^\ast \cdot  [y \Delta^{\frac{1}{2q}}], 
\end{split}
\]
so that $f_c$ and $f_d$ have the right properties.
 
 The claims about the supports now follow automatically. For the non-negativity, write $a = \lambda(f) \lambda(g), f,g \in C_c(G)$. Let $F \subseteq G$ be compact. We have $P_F j_1(a)P_F =  ( \lambda(f)^* \Delta^{\frac{1}{2}} P_F)^\ast  ( \lambda(g) \Delta^{\frac{1}{2}} P_F) = (\phi(t^{-1}s) )_{s,t \in F}$ for some $\phi \in A(G)$, see Lemma \ref{lemma=isom_AG_L1}.  But $( \lambda(f)^\ast \Delta^{\frac{1}{2}} P_F)^\ast  ( \lambda(g) \Delta^{\frac{1}{2}} P_F)_{s,t}  =  \Delta^{\frac{1}{2}}(st^{-1})   f_a(s t^{-1})$ is nonnegative.   
\end{proof}

\subsection*{Acknowledgement}
The authors wish to thank \'Eric Ricard for useful discussions on this paper, and the referee for useful comments.

\end{document}